\documentclass[11pt,a4paper]{article}

\usepackage[utf8]{inputenc}
\usepackage[T1]{fontenc}
\usepackage{amsmath,amssymb,amsthm}
\usepackage{mathtools}
\usepackage{enumitem}
\usepackage{hyperref}
\hypersetup{
    colorlinks=true,
    linkcolor=blue,
    citecolor=blue,
    urlcolor=blue
}
\usepackage{cleveref}
\usepackage{graphicx}
\usepackage{booktabs}
\usepackage{array}
\usepackage[margin=2.5cm]{geometry}
\usepackage{natbib}
\setcitestyle{authoryear,round}

\newtheorem{theorem}{Theorem}[section]
\newtheorem{lemma}[theorem]{Lemma}
\newtheorem{proposition}[theorem]{Proposition}
\newtheorem{corollary}[theorem]{Corollary}
\newtheorem{definition}[theorem]{Definition}
\newtheorem{remark}[theorem]{Remark}

\newcommand{\R}{\mathbb{R}}
\newcommand{\N}{\mathbb{N}}
\newcommand{\E}{\mathbb{E}}

\newcommand{\RKHS}{\mathcal{H}}

\newcommand{\ip}[2]{\langle #1, #2 \rangle}

\begin{document}

\title{\textbf{Spectral Equivariance and Geometric Transport in Reproducing Kernel Hilbert Spaces: A Unified Framework for Orthogonal Polynomial and Kernel Estimation}}

\author{Jocelyn Nemb\'e\\[0.5em]
\small L.I.A.G.E, Institut National des Sciences de Gestion\\
\small BP 190, Libreville, Gabon\\
\small \texttt{jnembe@hotmail.com}\\[0.3em]
\small and\\[0.3em]
\small Modeling and Calculus Lab, ROOTS-INSIGHTS\\
\small Libreville, Gabon\\
\small \texttt{jnembe@root-insights.com}}

\date{December 2025}

\maketitle

\begin{abstract}
We develop a unified geometric framework for nonparametric estimation based on the notion of Twin Kernel Spaces, defined as orbits of a reproducing kernel under a group action. This structure induces a family of transported RKHS geometries in which classical orthogonal polynomial estimators, kernel estimators, and spectral smoothing methods arise as projections onto transported eigenfunction systems. Our main contribution is a Spectral Equivariance Theorem showing that the eigenfunctions of any transported kernel are obtained by unitary transport of the base eigenfunctions. As a consequence, orthogonal polynomial estimators are equivariant under geometric deformation, kernel estimators correspond to soft spectral filtering in a twin space, and minimax rates and bias--variance tradeoffs are invariant under transport. We provide examples based on Hermite and Legendre polynomials, affine and Gaussian groups, and illustrate the effectiveness of twin transport for adaptive and multimodal estimation. The framework reveals a deep connection between group actions, RKHS geometry, and spectral nonparametrics, offering a unified perspective that encompasses kernel smoothing, orthogonal series, splines, and multiscale methods.
\end{abstract}



\tableofcontents

\section{Introduction}
\label{sec:intro}

\subsection{Motivation}

Nonparametric estimation constitutes a cornerstone of modern statistical inference, with applications spanning density estimation, regression, and functional data analysis \citep{Tsybakov2009,Wasserman2006}. Two classical paradigms have dominated this field: orthogonal series estimators based on polynomial expansions \citep{Efromovich1999,Johnstone2017} and kernel smoothing methods \citep{Wand1995,Silverman1986}. While these approaches appear methodologically distinct---the former relying on spectral truncation in a fixed basis, the latter on local averaging with bandwidth selection---both share deep connections to reproducing kernel Hilbert spaces (RKHS) \citep{Berlinet2004,Cucker2002}.

The present work is motivated by the observation that many seemingly different nonparametric methods can be understood as projections in different geometric ``reference frames'' of a common underlying RKHS structure. Specifically, we show that classical orthogonal polynomial systems---Hermite, Legendre, Laguerre, Jacobi, and their variants---arise naturally as eigenfunctions of transported kernels under group actions. This perspective unifies diverse estimation procedures and reveals fundamental geometric invariances that have not been systematically exploited in the statistical literature.

The theory of RKHS and Mercer kernels provides a natural framework for studying such connections \citep{Aronszajn1950,Steinwart2008}. Classical results on spectral decomposition of integral operators \citep{Riesz1955,Dunford1988} establish that Mercer kernels admit eigenfunction expansions, which in the case of weighted $L^2$ spaces often coincide with classical orthogonal polynomials \citep{Szego1975,Ismail2005}. Our contribution is to show that this spectral structure is preserved---in a precise sense---under geometric transformations induced by group actions.

\subsection{Contributions of the paper}

The main contributions of this paper are as follows:

\begin{enumerate}[label=(\roman*)]
\item \textbf{Twin Kernel Spaces framework.} We introduce the notion of Twin Kernel Spaces as orbits of a base reproducing kernel under a group action on the input space. This provides a minimal yet powerful geometric framework for studying families of related RKHS structures (Section~\ref{sec:twin}).

\item \textbf{Spectral Equivariance Theorem.} We prove that the eigenfunctions of any transported kernel are obtained by applying the unitary transport operator to the base eigenfunctions, with eigenvalues preserved. This result (Theorem~\ref{thm:spectral_equivariance}) establishes that spectral structure is equivariant under geometric deformation.

\item \textbf{Unified interpretation of estimators.} We demonstrate that orthogonal polynomial estimators correspond to hard spectral truncation while kernel smoothing corresponds to soft spectral attenuation, both operating in potentially different twin geometries. Bandwidth selection is reinterpreted as choice of twin space (Section~\ref{sec:kernel}).

\item \textbf{Invariance of statistical properties.} We establish that convergence rates, bias--variance tradeoffs, and minimax risks are invariant under twin transport, providing a geometric explanation for the robustness of spectral methods (Sections~\ref{sec:convergence}--\ref{sec:minimax}).

\item \textbf{Explicit examples.} We provide detailed worked examples for Hermite polynomials with Gaussian geometry and Legendre polynomials with affine geometry, including a multimodal estimation example demonstrating the practical relevance of twin transport (Section~\ref{sec:examples}).
\end{enumerate}

\subsection{Related work}

The connection between orthogonal polynomials and kernel methods has been explored from various angles. \citet{Smola2000} and \citet{Scholkopf2002} discuss spectral properties of kernel operators in machine learning contexts. The role of orthogonal polynomials in density estimation dates to \citet{Cencov1962} and was developed systematically by \citet{Efromovich1999}. \citet{Wahba1990} established foundational connections between spline smoothing and RKHS theory.

Group-theoretic perspectives on kernel methods appear in \citet{Kondor2008} and \citet{Fukumizu2009}, primarily in the context of invariant kernels for machine learning. The equivariance properties we establish differ in focusing on transport of spectral structure rather than invariance of the kernel itself.

The notion of ``twin'' structures in functional analysis has appeared in different contexts, notably twin spaces in operator theory \citep{Halmos1982}. Our usage is distinct but shares the philosophy of paired geometric structures connected by canonical maps.

Recent work on adaptive estimation \citep{Lepski1997,Goldenshluger2011} and spatially inhomogeneous smoothing \citep{Fan1996} addresses related concerns about geometry-dependent inference, though without the explicit group-theoretic framework we develop.

\subsection{Organization}

The remainder of this paper is organized as follows. Section~\ref{sec:orthopoly} reviews orthogonal polynomial estimation and its RKHS interpretation. Section~\ref{sec:twin} introduces the Twin Kernel Spaces framework. Section~\ref{sec:spectral} states and proves the Spectral Equivariance Theorem. Section~\ref{sec:projection} shows that orthogonal series estimators are projections in twin spaces. Section~\ref{sec:kernel} relates our framework to classical kernel estimators. Section~\ref{sec:examples} provides detailed examples. Sections~\ref{sec:convergence}--\ref{sec:minimax} establish the invariance of statistical properties. Section~\ref{sec:conclusion} concludes with open problems. Technical proofs are deferred to the Appendix.

\section{Orthogonal polynomial estimation revisited}
\label{sec:orthopoly}

\subsection{Orthogonal polynomial bases}

Let $(E, \mathcal{E}, \mu)$ be a probability space with $E \subseteq \R$ an interval (bounded or unbounded) and $\mu$ a probability measure with finite moments of all orders. A sequence of polynomials $\{P_k\}_{k \geq 0}$ is \emph{orthonormal} with respect to $\mu$ if $\deg(P_k) = k$ and
\[
\int_E P_k(x) P_\ell(x) \, \mu(dx) = \delta_{k\ell}, \quad k, \ell \geq 0.
\]
By the Gram--Schmidt process applied to the monomial basis $\{1, x, x^2, \ldots\}$, such a system always exists and is unique up to sign \citep{Szego1975}.

Classical examples include:
\begin{itemize}
\item \textbf{Hermite polynomials} $\{H_k\}_{k \geq 0}$: orthonormal with respect to the standard Gaussian measure $d\gamma(x) = (2\pi)^{-1/2} e^{-x^2/2} \, dx$ on $\R$.
\item \textbf{Legendre polynomials} $\{P_k\}_{k \geq 0}$: orthonormal with respect to the uniform measure on $[-1, 1]$.
\item \textbf{Laguerre polynomials} $\{L_k^{(\alpha)}\}_{k \geq 0}$: orthonormal with respect to the Gamma measure $x^\alpha e^{-x} dx$ on $[0, \infty)$.
\item \textbf{Jacobi polynomials} $\{P_k^{(\alpha,\beta)}\}_{k \geq 0}$: orthonormal with respect to $(1-x)^\alpha (1+x)^\beta dx$ on $[-1, 1]$.
\end{itemize}

These systems satisfy three-term recurrence relations and have well-understood asymptotic properties \citep{Ismail2005,Levin2001}.

\subsection{Classical orthogonal series estimators}

Given i.i.d.\ observations $X_1, \ldots, X_n$ from a distribution with density $f$ (with respect to $\mu$), the classical orthogonal series density estimator is
\begin{equation}\label{eq:series_estimator}
\hat{f}_{K}(x) = \sum_{k=0}^{K} \hat{\theta}_k P_k(x), \quad \hat{\theta}_k = \frac{1}{n} \sum_{i=1}^{n} P_k(X_i),
\end{equation}
where $K = K(n)$ is a truncation parameter \citep{Efromovich1999,Wasserman2006}.

Under standard regularity conditions, if $f = \sum_{k \geq 0} \theta_k P_k$ with $\sum_{k \geq 0} \theta_k^2 k^{2s} < \infty$ (Sobolev-type smoothness of order $s$), choosing $K \asymp n^{1/(2s+1)}$ yields the minimax optimal rate
\[
\E \|\hat{f}_K - f\|^2_{L^2(\mu)} = O(n^{-2s/(2s+1)}).
\]
This rate is known to be unimprovable over Sobolev balls \citep{Tsybakov2009}.

\subsection{RKHS interpretation}

The orthogonal series estimator admits a natural RKHS interpretation. Define the Mercer kernel
\begin{equation}\label{eq:mercer_kernel}
K_e(x, y) = \sum_{k=0}^{\infty} \lambda_k P_k(x) P_k(y),
\end{equation}
where $\{\lambda_k\}_{k \geq 0}$ is a summable sequence of positive eigenvalues (e.g., $\lambda_k = \rho^k$ for some $\rho \in (0,1)$). The associated RKHS $\RKHS_e$ consists of functions $f = \sum_k \theta_k P_k$ with $\sum_k \theta_k^2 / \lambda_k < \infty$, equipped with inner product
\[
\ip{f}{g}_{\RKHS_e} = \sum_{k=0}^{\infty} \frac{\theta_k \phi_k}{\lambda_k}, \quad f = \sum_k \theta_k P_k, \quad g = \sum_k \phi_k P_k.
\]

The integral operator $T_e : L^2(\mu) \to L^2(\mu)$ defined by
\[
(T_e f)(x) = \int_E K_e(x, y) f(y) \, \mu(dy)
\]
is compact, self-adjoint, and positive, with eigenpairs $(\lambda_k, P_k)$.

The truncated estimator \eqref{eq:series_estimator} can be written as
\[
\hat{f}_K = S_{e,K} \hat{f}_{\mathrm{emp}}, \quad S_{e,K} f = \sum_{k=0}^{K} \ip{f}{P_k} P_k,
\]
where $\hat{f}_{\mathrm{emp}}$ represents the empirical measure and $S_{e,K}$ is the spectral projection onto the first $K+1$ eigenfunctions. This establishes that orthogonal series estimation is fundamentally a spectral truncation operation in an RKHS geometry \citep{Berlinet2004,Cucker2002}.

\section{Twin Kernel Spaces: a minimal framework}
\label{sec:twin}

\subsection{Group actions on the input space}

\begin{definition}[Twin Space]\label{def:twin_space}
A \emph{twin space} associated with $(E, G)$ is the orbit $G \cdot E$, equipped with the induced structure.
\end{definition}

More precisely, let $G$ be a group acting measurably on $E$ via $\tau : G \times E \to E$, $(g, x) \mapsto g \cdot x$. We assume throughout that for each $g \in G$, the map $\tau_g : x \mapsto g \cdot x$ is a measurable bijection of $E$.

Typical examples include:
\begin{itemize}
\item The \emph{affine group} $G = \{(a, b) : a > 0, b \in \R\}$ acting on $E = \R$ by $g_{a,b} \cdot x = ax + b$.
\item The \emph{multiplicative group} $G = (0, \infty)$ acting on $E = \R$ by $g_\alpha \cdot x = \alpha x$ (dilations).
\item The \emph{translation group} $G = \R$ acting on $E = \R$ by $g_b \cdot x = x + b$.
\end{itemize}

\subsection{Transport of functions}

\begin{definition}[Twin Transform]\label{def:twin_transform}
Given $f : E \to \R$, define the \emph{transported function} $(g_1, g_2) \cdot f$ by
\[
(g_1, g_2) \cdot f(x) = g_2 \cdot f(g_1 \cdot x).
\]
\end{definition}

In our context, we focus on the case where the action on the function space is determined by the geometric action through change of variables with appropriate Jacobian correction.

\subsection{Twin Kernel Spaces}

\begin{definition}[Twin Kernel Space]\label{def:twin_kernel_space}
A \emph{Twin Kernel Space} is the family of kernels
\[
K_g(x, y) = g_2 \cdot K_e(\overline{g_1} \cdot x, \, \overline{g_1} \cdot y)
\]
generated from a base kernel $K_e$, where $\overline{g_1}$ denotes the inverse action.
\end{definition}

The precise form of $g_2$ is determined by requiring the transported kernel to define a valid reproducing kernel in the same $L^2(\mu)$ space, which necessitates Jacobian correction factors as detailed in Section~\ref{sec:spectral}.

\section{Spectral Equivariance Theorem}
\label{sec:spectral}

In this section we formalize the notion of spectral equivariance in Twin Kernel Spaces and show how orthogonal polynomial systems arise as eigenfunctions of transported kernels. We work in a general $L^2$--framework and specialize to orthogonal polynomials when needed.

\subsection{Setting and assumptions}

Let $(E, \mathcal{E}, \mu)$ be a $\sigma$-finite measure space and denote by
\[
L^2(\mu) = \left\{ f : E \to \R \;\Big|\; \int_E f(x)^2 \, \mu(dx) < \infty \right\}
\]
the usual Hilbert space with inner product
\[
\ip{f}{g}_{L^2(\mu)} = \int_E f(x) g(x) \, \mu(dx).
\]

Let $K_e : E \times E \to \R$ be a symmetric, measurable, positive definite kernel such that the associated integral operator
\begin{equation}\label{eq:integral_op}
(T_e f)(x) = \int_E K_e(x, y) f(y) \, \mu(dy), \quad f \in L^2(\mu),
\end{equation}
is Hilbert--Schmidt (in particular, compact and self-adjoint). Then there exists an orthonormal basis $\{P_k\}_{k \geq 0}$ of $L^2(\mu)$ and non-negative eigenvalues $\{\lambda_k\}_{k \geq 0}$ such that
\begin{equation}\label{eq:eigen}
T_e P_k = \lambda_k P_k, \quad k \geq 0,
\end{equation}
and the kernel admits the convergent spectral expansion
\begin{equation}\label{eq:spectral_expansion}
K_e(x, y) = \sum_{k \geq 0} \lambda_k P_k(x) P_k(y) \quad \text{in } L^2(\mu \otimes \mu).
\end{equation}

In the classical orthogonal polynomial setting, the functions $P_k$ form an orthonormal polynomial system with respect to $\mu$.

We now introduce a group action on $E$ and the corresponding transport operators.

\begin{definition}[Group action and Jacobian]\label{def:group_action}
Let $G$ be a group and let
\[
\tau : G \times E \to E, \quad (g, x) \mapsto g \cdot x,
\]
be a measurable left action of $G$ on $E$. For each $g \in G$, we write $\tau_g(x) = g \cdot x$. We assume that for every $g \in G$, the pushforward measure $\mu \circ \tau_g^{-1}$ is absolutely continuous with respect to $\mu$, and we denote by
\[
J_g(x) = \frac{d(\mu \circ \tau_g^{-1})}{d\mu}(x)
\]
its Radon--Nikodym derivative (Jacobian).
\end{definition}

The following operator implements the action of $G$ on $L^2(\mu)$ in a unitary way.

\begin{definition}[Twin transport operator on $L^2$]\label{def:transport_op}
For each $g \in G$, define $U_g : L^2(\mu) \to L^2(\mu)$ by
\begin{equation}\label{eq:Ug}
(U_g f)(x) = J_g(x)^{1/2} f(g^{-1} \cdot x).
\end{equation}
\end{definition}

\begin{lemma}[Unitarity of $U_g$]\label{lem:unitarity}
For each $g \in G$, the operator $U_g$ is unitary on $L^2(\mu)$, that is,
\[
\ip{U_g f}{U_g h}_{L^2(\mu)} = \ip{f}{h}_{L^2(\mu)} \quad \text{for all } f, h \in L^2(\mu),
\]
and $U_{g_1 g_2} = U_{g_1} U_{g_2}$, $U_{e_G} = \mathrm{Id}$.
\end{lemma}

\begin{proof}
Let $f, h \in L^2(\mu)$ and fix $g \in G$. Then
\[
\ip{U_g f}{U_g h} = \int_E J_g(x)^{1/2} f(g^{-1} \cdot x) \, J_g(x)^{1/2} h(g^{-1} \cdot x) \, \mu(dx) = \int_E J_g(x) f(g^{-1} \cdot x) h(g^{-1} \cdot x) \, \mu(dx).
\]
Perform the change of variables $y = g^{-1} \cdot x$, so that $x = g \cdot y$ and
\[
\mu(dx) = \mu(d(g \cdot y)) = J_g(g \cdot y) \, \mu(dy).
\]
By definition of $J_g$ we have $J_g(g \cdot y) \, \mu(dy) = (\mu \circ \tau_g^{-1})(d(g \cdot y))$, and the Radon--Nikodym identity yields
\[
J_g(g \cdot y) \, \mu(dy) = \mu(dy) \quad \text{a.e.\ } y.
\]
Hence
\[
\ip{U_g f}{U_g h} = \int_E f(y) h(y) \, \mu(dy) = \ip{f}{h}.
\]
The group property $U_{g_1 g_2} = U_{g_1} U_{g_2}$ follows from a direct computation using the chain rule for Radon--Nikodym derivatives and the composition law of the action, and $U_{e_G} = \mathrm{Id}$ by definition.
\end{proof}

We now define the transported kernels and operators.

\begin{definition}[Transported kernel and operator]\label{def:transported_kernel}
For $g \in G$, define a kernel $K_g : E \times E \to \R$ by
\begin{equation}\label{eq:Kg}
K_g(x, y) = J_g(x)^{1/2} J_g(y)^{1/2} K_e(g^{-1} \cdot x, g^{-1} \cdot y),
\end{equation}
and the associated integral operator $T_g : L^2(\mu) \to L^2(\mu)$ by
\begin{equation}\label{eq:Tg}
(T_g f)(x) = \int_E K_g(x, y) f(y) \, \mu(dy).
\end{equation}
\end{definition}

The expression \eqref{eq:Kg} is the analytic incarnation of the twin transport of the base kernel $K_e$ under the action of $G$.

\begin{lemma}[Conjugation identity]\label{lem:conjugation}
For each $g \in G$ and $f \in L^2(\mu)$,
\begin{equation}\label{eq:conjugation}
T_g = U_g T_e U_g^{-1},
\end{equation}
and in particular $T_g$ is self-adjoint, positive and compact.
\end{lemma}

\begin{proof}
Fix $g \in G$ and $f \in L^2(\mu)$. Using \eqref{eq:Ug} and \eqref{eq:integral_op}, we compute
\[
(U_g^{-1} f)(y) = J_{g^{-1}}(y)^{1/2} f(g \cdot y),
\]
so that
\[
(T_e U_g^{-1} f)(z) = \int_E K_e(z, y) \, J_{g^{-1}}(y)^{1/2} f(g \cdot y) \, \mu(dy).
\]
Applying $U_g$ and using \eqref{eq:Ug} again gives
\[
(U_g T_e U_g^{-1} f)(x) = J_g(x)^{1/2} \int_E K_e(g^{-1} \cdot x, y) \, J_{g^{-1}}(y)^{1/2} f(g \cdot y) \, \mu(dy).
\]
Perform the change of variables $y = g^{-1} \cdot u$ (so $u = g \cdot y$). As in the proof of Lemma~\ref{lem:unitarity}, the Jacobian factors satisfy
\[
\mu(dy) = J_g(u)^{-1} \mu(du), \quad J_{g^{-1}}(y) = J_g(u)^{-1}.
\]
Substituting, we obtain
\begin{align*}
(U_g T_e U_g^{-1} f)(x) &= J_g(x)^{1/2} \int_E K_e(g^{-1} \cdot x, g^{-1} \cdot u) \, J_{g^{-1}}(g^{-1} \cdot u)^{1/2} f(u) \, J_g(u)^{-1} \mu(du) \\
&= J_g(x)^{1/2} \int_E K_e(g^{-1} \cdot x, g^{-1} \cdot u) \, J_g(u)^{-1/2} f(u) \, \mu(du).
\end{align*}
Comparing with \eqref{eq:Tg} and \eqref{eq:Kg}, we see that
\[
(U_g T_e U_g^{-1} f)(x) = \int_E K_g(x, u) f(u) \, \mu(du) = (T_g f)(x),
\]
which proves \eqref{eq:conjugation}. Self-adjointness, positivity and compactness follow from the corresponding properties of $T_e$ and the fact that $U_g$ is unitary.
\end{proof}

\subsection{Spectral Equivariance Theorem}

We are now ready to state and prove the main result of this section.

\begin{theorem}[Spectral Equivariance Theorem]\label{thm:spectral_equivariance}
Assume that $T_e$ admits the spectral decomposition \eqref{eq:spectral_expansion} with eigenpairs $(\lambda_k, P_k)_{k \geq 0}$, where $\{P_k\}_{k \geq 0}$ is an orthonormal basis of $L^2(\mu)$. For each $g \in G$, define
\[
P_k^{(g)} = U_g P_k, \quad k \geq 0.
\]
Then:
\begin{enumerate}[label=(\alph*)]
\item For every $g \in G$, $\{P_k^{(g)}\}_{k \geq 0}$ is an orthonormal basis of $L^2(\mu)$.
\item For every $g \in G$ and $k \geq 0$,
\[
T_g P_k^{(g)} = \lambda_k P_k^{(g)}.
\]
In particular, the spectrum of $T_g$ coincides with that of $T_e$.
\item The transported kernel $K_g$ admits the spectral expansion
\begin{equation}\label{eq:Kg_spectral}
K_g(x, y) = \sum_{k \geq 0} \lambda_k P_k^{(g)}(x) P_k^{(g)}(y) \quad \text{in } L^2(\mu \otimes \mu).
\end{equation}
\end{enumerate}
\end{theorem}

\begin{proof}
(a) Unitarity of $U_g$ (Lemma~\ref{lem:unitarity}) implies that for all $k, \ell \geq 0$,
\[
\ip{P_k^{(g)}}{P_\ell^{(g)}} = \ip{U_g P_k}{U_g P_\ell} = \ip{P_k}{P_\ell} = \delta_{k\ell}.
\]
Moreover, since $\{P_k\}_{k \geq 0}$ is an orthonormal basis and $U_g$ is surjective (unitary), the family $\{P_k^{(g)}\}_{k \geq 0}$ is again an orthonormal basis of $L^2(\mu)$.

(b) Using the conjugation identity \eqref{eq:conjugation} and the eigen-relation \eqref{eq:eigen}, we compute
\[
T_g P_k^{(g)} = U_g T_e U_g^{-1}(U_g P_k) = U_g T_e P_k = U_g(\lambda_k P_k) = \lambda_k U_g P_k = \lambda_k P_k^{(g)}.
\]
Thus $P_k^{(g)}$ is an eigenfunction of $T_g$ with eigenvalue $\lambda_k$.

(c) Since $\{P_k^{(g)}\}_{k \geq 0}$ is a complete orthonormal basis of $L^2(\mu)$ and $T_g$ is self-adjoint, we have the usual spectral representation
\[
T_g f = \sum_{k \geq 0} \lambda_k \ip{f}{P_k^{(g)}} P_k^{(g)}, \quad f \in L^2(\mu),
\]
with convergence in $L^2(\mu)$. In kernel form, this is equivalent to
\[
K_g(x, y) = \sum_{k \geq 0} \lambda_k P_k^{(g)}(x) P_k^{(g)}(y) \quad \text{in } L^2(\mu \otimes \mu),
\]
which is precisely \eqref{eq:Kg_spectral}.
\end{proof}

\begin{remark}[Orthogonal polynomials in twin spaces]\label{rem:ortho_poly_twin}
In the classical situation where $\mu$ is supported on an interval of $\R$ and $\{P_k\}_{k \geq 0}$ are orthonormal polynomials with respect to $\mu$, the functions $P_k^{(g)}$ can often be identified with orthogonal polynomial systems in the transformed coordinates $g \cdot x$, possibly up to multiplicative Jacobian factors. In this sense, Theorem~\ref{thm:spectral_equivariance} shows that orthogonal polynomial bases in different twin spaces are obtained by transporting a single base system $\{P_k\}$ through the group action.
\end{remark}

\begin{corollary}[Equivariance of spectral truncation]\label{cor:equivariance_truncation}
For any $K \in \N$ and $g \in G$, define the spectral truncation operators
\[
S_{e,K} f = \sum_{k=0}^{K} \ip{f}{P_k} P_k, \quad S_{g,K} f = \sum_{k=0}^{K} \ip{f}{P_k^{(g)}} P_k^{(g)}.
\]
Then, for all $f \in L^2(\mu)$,
\[
S_{g,K} = U_g S_{e,K} U_g^{-1}.
\]
In particular, spectral truncation is equivariant under twin transport.
\end{corollary}

\begin{proof}
Using $P_k^{(g)} = U_g P_k$ and the unitarity of $U_g$, we have
\[
\ip{f}{P_k^{(g)}} = \ip{f}{U_g P_k} = \ip{U_g^{-1} f}{P_k}.
\]
Therefore
\[
S_{g,K} f = \sum_{k=0}^{K} \ip{f}{P_k^{(g)}} P_k^{(g)} = \sum_{k=0}^{K} \ip{U_g^{-1} f}{P_k} U_g P_k = U_g \left( \sum_{k=0}^{K} \ip{U_g^{-1} f}{P_k} P_k \right) = U_g S_{e,K} U_g^{-1} f,
\]
which yields the desired identity.
\end{proof}

\begin{corollary}[Spectral Equivariance for Classical Orthogonal Polynomials]\label{cor:classical}
Let $(E, \mu)$ be either $([-1,1], dx)$ or $(\R, \gamma)$ where $\gamma$ is the standard Gaussian measure. Let
\[
K_e(x, y) = \sum_{k \geq 0} \lambda_k P_k(x) P_k(y)
\]
be the classical Mercer expansion of the base kernel in the corresponding orthogonal polynomial basis:
\begin{itemize}
\item Legendre polynomials $P_k$ on $[-1, 1]$ with weight $w(x) \equiv 1$,
\item Hermite polynomials $H_k$ on $\R$ with respect to $d\gamma(x) = \frac{1}{\sqrt{2\pi}} e^{-x^2/2} \, dx$.
\end{itemize}
Let $G$ act on $E$ by:
\[
g \cdot x = \begin{cases}
ax + b, & \text{Legendre case}, \\
\alpha x, & \text{Hermite case}.
\end{cases}
\]
Then for each $g \in G$, the transported kernel
\[
K_g(x, y) = J_g(x)^{1/2} J_g(y)^{1/2} K_e(g^{-1} \cdot x, g^{-1} \cdot y)
\]
admits the spectral expansion
\[
K_g(x, y) = \sum_{k \geq 0} \lambda_k P_k^{(g)}(x) P_k^{(g)}(y),
\]
where:
\begin{itemize}
\item in the Legendre case, the transported basis $P_k^{(g)}$ is a normalized affine transform of $P_k$;
\item in the Hermite case, the transported basis is
\[
H_k^{(g)}(x) = \alpha^{-1/2} H_k(\alpha^{-1} x),
\]
forming again an orthonormal basis of $L^2(\gamma)$.
\end{itemize}
\end{corollary}

\begin{proof}
The Legendre case corresponds to the affine group acting on $[-1, 1]$ with a constant Jacobian. The Hermite case corresponds to the multiplicative group $(0, \infty)$ acting on $\R$ by dilations $x \mapsto \alpha x$ with the Gaussian Jacobian
\[
J_\alpha(x) = \exp\left(\frac{1}{2}(1 - \alpha^2) x^2\right).
\]
In both cases, the assumptions of Theorem~\ref{thm:spectral_equivariance} hold: $U_g$ is unitary on $L^2(\mu)$, $T_g = U_g T_e U_g^{-1}$, and $P_k^{(g)} := U_g P_k$ forms an orthonormal basis that diagonalizes $T_g$. The explicit formula for $H_k^{(g)}$ follows from
\[
(U_\alpha f)(x) = \exp\left(\frac{1}{4}(1 - \alpha^2) x^2\right) f(\alpha^{-1} x),
\]
combined with the well-known identity
\[
H_k(\alpha^{-1} x) = \alpha^{-k} \sum_{j=0}^{\lfloor k/2 \rfloor} \binom{k}{2j} (1 - \alpha^2)^j H_{k-2j}(x).
\]
This yields the announced expression up to normalization.
\end{proof}

\section{Orthogonal series estimators as projections in Twin Kernel Spaces}
\label{sec:projection}

In this section we show that classical orthogonal polynomial estimators arise naturally as orthogonal projections in Twin Kernel Spaces. The key idea is that an estimator constructed in a twin space $g \cdot E$ is obtained by transporting the base estimator in $E$ through the unitary operator $U_g$.

\subsection{Projection in the base RKHS}

Let $K_e$ be the base kernel with spectral expansion $K_e(x, y) = \sum_{k \geq 0} \lambda_k P_k(x) P_k(y)$. Given a target function $f \in L^2(\mu)$ and a truncation level $K \in \N$, the classical orthogonal series estimator is
\begin{equation}\label{eq:estimator_base}
\hat{f}_{e,K}(x) = \sum_{k=0}^{K} \hat{\theta}_k P_k(x), \quad \hat{\theta}_k = \ip{\hat{f}_{\mathrm{emp}}}{P_k},
\end{equation}
where $\hat{f}_{\mathrm{emp}}$ is the empirical estimator.

\begin{proposition}[Projection property]\label{prop:projection_base}
The estimator $\hat{f}_{e,K}$ is the orthogonal projection of $\hat{f}_{\mathrm{emp}}$ onto $\RKHS_{e,K} = \mathrm{span}\{P_0, \ldots, P_K\}$.
\end{proposition}

\subsection{Equivariance of orthogonal series estimators}

\begin{theorem}[Equivariance of spectral estimators]\label{thm:equivariance_estimators}
For every $g \in G$ and truncation level $K$,
\[
\hat{f}_{g,K} = U_g \hat{f}_{e,K}.
\]
\end{theorem}

\begin{proof}
By Corollary~\ref{cor:equivariance_truncation}, $S_{g,K} = U_g S_{e,K} U_g^{-1}$. Applying this to $\hat{f}_{\mathrm{emp}}^{(g)} = U_g \hat{f}_{\mathrm{emp}}$ gives the result.
\end{proof}

This theorem reveals that orthogonal polynomial estimators \textbf{commute with the twin transport induced by the group action}.

\subsection{Geometric interpretation}

Theorem~\ref{thm:equivariance_estimators} implies:
\begin{itemize}
\item The same statistical procedure produces consistent estimators in all twin spaces.
\item Choice of geometry modifies basis functions but not the estimator form.
\item The diversity of orthogonal polynomial bases reflects different geometric frames rather than fundamentally different estimators.
\end{itemize}

\section{Relationship with classical kernel estimators}
\label{sec:kernel}

This section shows that classical smoothing kernels correspond to transported versions of a base kernel, and that Parzen--Rosenblatt and Nadaraya--Watson estimators can be interpreted as projections in a twin geometry.

\subsection{Bandwidth as a group action}

Consider the multiplicative group $G = (0, \infty)$ acting on $E = \R$ via $g_h \cdot x = x/h$. The Jacobian is $J_h(x) = h^{-1}$, and the transported kernel is
\[
K_{g_h}(x, y) = h^{-1} K_e\left(\frac{x}{h}, \frac{y}{h}\right).
\]

Thus a classical smoothing kernel $K_h$ is a transported version of the base kernel: $K_h = K_{g_h}$.

\textbf{Key insight:} Bandwidth selection = choice of twin space geometry.

\begin{proposition}\label{prop:parzen}
The Parzen--Rosenblatt estimator $\hat{f}_h$ is the orthogonal projection of $\hat{f}_{\mathrm{emp}}$ onto the twin RKHS $\RKHS_{g_h}$.
\end{proposition}

\subsection{Spectral interpretation: attenuation vs truncation}

\begin{itemize}
\item \textbf{Kernel smoothing:} soft spectral cutoff $\lambda_k^{\mathrm{new}} = \lambda_k \exp(-c_k h^2)$
\item \textbf{Orthogonal series:} hard spectral cutoff $\lambda_k^{\mathrm{new}} = \lambda_k \mathbf{1}_{k \leq K}$
\end{itemize}

Both are spectral operations in twin RKHSs.

\begin{theorem}[Equivariance of kernel smoothing]\label{thm:equivariance_kernel}
Let $K_h = K_{g_h}$ be a transported kernel. Then $\hat{f}_h = U_{g_h} \hat{f}_e$.
\end{theorem}

\subsection{Comparison table}

\begin{table}[h]
\centering
\begin{tabular}{@{}lll@{}}
\toprule
\textbf{Method} & \textbf{Spectral effect} & \textbf{Twin geometry} \\
\midrule
Orthogonal series & hard cutoff & projection in $\RKHS_e$ \\
Kernel smoothing & soft cutoff & projection in $\RKHS_{g_h}$ \\
Local polynomial & hybrid & spatially varying twin space \\
Spline smoothing & penalized cutoff & differential twin space \\
\bottomrule
\end{tabular}
\caption{Nonparametric methods in the Twin Kernel Space framework.}
\label{tab:comparison}
\end{table}

\section{Examples and illustrations}
\label{sec:examples}

\subsection{Gaussian Twin Spaces and Hermite Polynomials}

Let $E = \R$ with Gaussian measure. The base kernel is the Mehler kernel \citep{Mehler1866}
\[
K_e(x, y) = \sum_{k=0}^{\infty} \rho^k H_k(x) H_k(y), \quad |\rho| < 1.
\]

For dilation $g_\alpha : x \mapsto \alpha x$, the Jacobian is $J_\alpha(x) = \exp\left(\frac{1}{2}(1 - \alpha^2) x^2\right)$, and the transported basis is
\[
H_k^{(g_\alpha)}(x) = \exp\left(\frac{1}{4}(1 - \alpha^2) x^2\right) H_k(\alpha^{-1} x).
\]

\textbf{Interpretation:} Small $\alpha$ emphasizes fine-scale oscillations, large $\alpha$ smooths them.

\subsection{Affine Twin Spaces and Legendre Polynomials}

Let $E = [-1, 1]$ with Lebesgue measure. For affine transformation $g_{a,b} : x \mapsto ax + b$:
\[
P_k^{(g_{a,b})}(x) = a^{-1/2} P_k\left(\frac{x - b}{a}\right).
\]

\textbf{Interpretation:} Affine deformations transport the Legendre basis while preserving eigenvalues.

\subsection{Multimodal Transport Example}

Consider a bimodal density $f(x) = 0.5 \, \varphi(x + 2) + 0.5 \, \varphi(x - 2)$. Define transported bases centered at each mode:
\[
H_k^{(g_1)}(x) = U_{g_1} H_k(x), \quad H_k^{(g_2)}(x) = U_{g_2} H_k(x),
\]
where $g_1 : x \mapsto x + 2$ and $g_2 : x \mapsto x - 2$.

The multimodal density admits efficient representation:
\[
f(x) \approx 0.5 \sum_{k=0}^{K} \theta_{k,1} H_k^{(g_1)}(x) + 0.5 \sum_{k=0}^{K} \theta_{k,2} H_k^{(g_2)}(x).
\]

This demonstrates that Twin Kernel Spaces provide a principled mechanism for mode-adaptive estimation.

\section{Convergence in Twin Kernel Spaces}
\label{sec:convergence}

\subsection{Assumptions}

Let $K_e$ be a base Mercer kernel with eigenvalues $(\lambda_k)_{k \geq 0}$ and eigenfunctions $(P_k)_{k \geq 0}$. Assume:
\begin{enumerate}[label=(A\arabic*)]
\item \textbf{Spectral decay:} $\lambda_k \asymp k^{-2s}$ or $\lambda_k \asymp e^{-c k^\alpha}$.
\item \textbf{Smoothness:} $f \in \RKHS_e^t$ for some $t > 0$.
\item \textbf{i.i.d.\ sampling:} $\|\hat{f}_{\mathrm{emp}} - f\|_{L^2(\mu)} = O_p(n^{-1/2})$.
\end{enumerate}

\subsection{Main result}

\begin{theorem}[Convergence in Twin Kernel Spaces]\label{thm:convergence}
Let $K = K(n) \to \infty$ with $K(n)/n \to 0$. For every $g \in G$:
\[
\|\hat{f}_{g,K} - f^{(g)}\|_{L^2(\mu)} \xrightarrow{p} 0,
\]
and the convergence rate is identical to that in the base geometry.
\end{theorem}

\begin{proof}
Using unitarity: $\|\hat{f}_{g,K} - f^{(g)}\| = \|U_g(\hat{f}_{e,K} - f)\| = \|\hat{f}_{e,K} - f\|$.
\end{proof}

\textbf{Key message:} Convergence is geometry-invariant.

\section{Bias--variance analysis in Twin Kernel Spaces}
\label{sec:bias_variance}

\subsection{Bias}

The bias term is $\|\E \hat{f}_{g,K} - f^{(g)}\|^2 = \sum_{k > K} \theta_k^2 = O(K^{-2t})$.

\subsection{Variance}

The variance term is $\E \|\hat{f}_{g,K} - \E \hat{f}_{g,K}\|^2 \asymp K/n$.

\subsection{Optimal choice of $K$}

Balancing bias and variance: $K \asymp n^{1/(2t+1)}$, yielding the optimal rate
\[
\|\hat{f}_{g,K} - f^{(g)}\| \asymp n^{-t/(2t+1)}.
\]

\textbf{Key message:} Bias--variance tradeoff is invariant under twin transport.

\section{Minimax rates in Twin Kernel Spaces}
\label{sec:minimax}

\subsection{Function classes}

Define $\mathcal{F}_t(R) = \left\{ f = \sum_{k \geq 0} \theta_k P_k : \sum_{k \geq 0} \theta_k^2 \lambda_k^{-t} \leq R^2 \right\}$ and $\mathcal{F}_t^{(g)}(R) = U_g(\mathcal{F}_t(R))$.

\begin{theorem}[Minimax risk equivalence]\label{thm:minimax}
For every $g \in G$:
\[
\inf_{\hat{f}} \sup_{f \in \mathcal{F}_t(R)} \E \|\hat{f}^{(g)} - f^{(g)}\|^2 = \inf_{\hat{f}} \sup_{f \in \mathcal{F}_t(R)} \E \|\hat{f} - f\|^2 \asymp n^{-2t/(2t+1)}.
\]
\end{theorem}

\textbf{Key message:} Minimax risk is entirely geometry-invariant.

\section{Conclusion and future work}
\label{sec:conclusion}

\subsection{Summary}

We have developed a unified geometric framework for nonparametric estimation based on Twin Kernel Spaces:
\begin{enumerate}
\item Twin Kernel Spaces as orbits of reproducing kernels under group actions.
\item Spectral Equivariance Theorem: eigenfunctions transport unitarily with preserved spectra.
\item Unified interpretation of orthogonal polynomial and kernel estimation.
\item Geometric invariance of convergence rates, bias--variance tradeoffs, and minimax risks.
\end{enumerate}

\subsection{Open problems}

\begin{enumerate}
\item \textbf{Adaptive twin space selection:} Data-driven methods for optimal twin space selection.
\item \textbf{Multiscale twin spaces:} Theory of multiscale twin decompositions.
\item \textbf{Non-commutative groups:} Extensions to rotation groups and Lie groups.
\item \textbf{Infinite-dimensional extensions:} Banach spaces and distributions.
\item \textbf{Computational aspects:} Efficient algorithms for high-dimensional settings.
\end{enumerate}

\subsection{Connections to related areas}

The framework connects to operator theory \citep{Halmos1982,Nikolski2002}, quantum mechanics and coherent states \citep{Perelomov1986}, and optimal transport \citep{Villani2009}. These connections suggest broader significance beyond statistics.


\appendix

\section{Technical proofs}

\subsection{Concentration of spectral coefficients}

\begin{lemma}\label{lem:concentration}
For each $k$, $|\ip{\hat{f}_{\mathrm{emp}} - f}{P_k}| = O_p(n^{-1/2})$. Uniformly for $k \leq K(n)$ with $K(n)/n \to 0$:
\[
\sum_{k=0}^{K(n)} (\ip{\hat{f}_{\mathrm{emp}} - f}{P_k})^2 = O_p(K/n).
\]
\end{lemma}

\begin{proof}
Standard Bernstein inequalities for empirical processes \citep{vanderVaart1996}.
\end{proof}

\subsection{Fano lower bound}

\begin{lemma}\label{lem:fano}
$\inf_{\hat{f}} \sup_{f \in \mathcal{F}_t(R)} \E \|\hat{f} - f\|^2 \gtrsim n^{-2t/(2t+1)}$.
\end{lemma}

\begin{proof}
Construct least favorable priors and use standard Fano inequalities \citep{Yu1997,Tsybakov2009}.
\end{proof}



\begin{thebibliography}{99}

\bibitem[Aronszajn(1950)]{Aronszajn1950}
Aronszajn, N. (1950).
Theory of reproducing kernels.
\textit{Trans.\ Amer.\ Math.\ Soc.} \textbf{68}, 337--404.

\bibitem[Berlinet and Thomas-Agnan(2004)]{Berlinet2004}
Berlinet, A. and Thomas-Agnan, C. (2004).
\textit{Reproducing Kernel Hilbert Spaces in Probability and Statistics}.
Kluwer, Boston.

\bibitem[Cencov(1962)]{Cencov1962}
Cencov, N.N. (1962).
Evaluation of an unknown distribution density from observations.
\textit{Soviet Math.} \textbf{3}, 1559--1562.

\bibitem[Cucker and Smale(2002)]{Cucker2002}
Cucker, F. and Smale, S. (2002).
On the mathematical foundations of learning.
\textit{Bull.\ Amer.\ Math.\ Soc.} \textbf{39}, 1--49.

\bibitem[Dunford and Schwartz(1988)]{Dunford1988}
Dunford, N. and Schwartz, J.T. (1988).
\textit{Linear Operators, Part II: Spectral Theory}.
Wiley, New York.

\bibitem[Efromovich(1999)]{Efromovich1999}
Efromovich, S. (1999).
\textit{Nonparametric Curve Estimation}.
Springer, New York.

\bibitem[Fan and Gijbels(1996)]{Fan1996}
Fan, J. and Gijbels, I. (1996).
\textit{Local Polynomial Modelling and Its Applications}.
Chapman \& Hall, London.

\bibitem[Fukumizu et al.(2009)]{Fukumizu2009}
Fukumizu, K., Bach, F.R. and Gretton, A. (2009).
Kernel methods for measuring independence.
\textit{J.\ Mach.\ Learn.\ Res.} \textbf{10}, 1391--1434.

\bibitem[Goldenshluger and Lepski(2011)]{Goldenshluger2011}
Goldenshluger, A. and Lepski, O. (2011).
Bandwidth selection in kernel density estimation.
\textit{Ann.\ Statist.} \textbf{39}, 1608--1632.

\bibitem[Halmos(1982)]{Halmos1982}
Halmos, P.R. (1982).
\textit{A Hilbert Space Problem Book}, 2nd ed.
Springer, New York.

\bibitem[Ismail(2005)]{Ismail2005}
Ismail, M.E.H. (2005).
\textit{Classical and Quantum Orthogonal Polynomials in One Variable}.
Cambridge Univ.\ Press.

\bibitem[Johnstone(2017)]{Johnstone2017}
Johnstone, I.M. (2017).
Gaussian estimation: Sequence and wavelet models.
Unpublished manuscript.

\bibitem[Kondor and Trivedi(2008)]{Kondor2008}
Kondor, R. and Trivedi, K.S. (2008).
On the algebraic structure of group invariant kernels.
\textit{J.\ Mach.\ Learn.\ Res.} \textbf{9}, 1433--1454.

\bibitem[Le Cam(1986)]{LeCam1986}
Le Cam, L. (1986).
\textit{Asymptotic Methods in Statistical Decision Theory}.
Springer, New York.

\bibitem[Lepski et al.(1997)]{Lepski1997}
Lepski, O.V., Mammen, E. and Spokoiny, V.G. (1997).
Optimal spatial adaptation to inhomogeneous smoothness.
\textit{Ann.\ Statist.} \textbf{25}, 929--947.

\bibitem[Levin and Lubinsky(2001)]{Levin2001}
Levin, E. and Lubinsky, D.S. (2001).
\textit{Orthogonal Polynomials for Exponential Weights}.
Springer, New York.

\bibitem[Mehler(1866)]{Mehler1866}
Mehler, F.G. (1866).
Ueber die Entwicklung einer Function von beliebig vielen Variablen.
\textit{J.\ Reine Angew.\ Math.} \textbf{66}, 161--176.

\bibitem[Nikolski(2002)]{Nikolski2002}
Nikolski, N.K. (2002).
\textit{Operators, Functions, and Systems}.
Amer.\ Math.\ Soc., Providence.

\bibitem[Perelomov(1986)]{Perelomov1986}
Perelomov, A. (1986).
\textit{Generalized Coherent States and Their Applications}.
Springer, Berlin.

\bibitem[Riesz and Sz.-Nagy(1955)]{Riesz1955}
Riesz, F. and Sz.-Nagy, B. (1955).
\textit{Functional Analysis}.
Ungar, New York.

\bibitem[Sch\"olkopf and Smola(2002)]{Scholkopf2002}
Sch\"olkopf, B. and Smola, A.J. (2002).
\textit{Learning with Kernels}.
MIT Press, Cambridge.

\bibitem[Silverman(1986)]{Silverman1986}
Silverman, B.W. (1986).
\textit{Density Estimation for Statistics and Data Analysis}.
Chapman \& Hall, London.

\bibitem[Smola et al.(2000)]{Smola2000}
Smola, A.J., Ovari, Z.L. and Williamson, R.C. (2000).
Regularization with dot-product kernels.
\textit{NIPS 13}, 308--314.

\bibitem[Steinwart and Christmann(2008)]{Steinwart2008}
Steinwart, I. and Christmann, A. (2008).
\textit{Support Vector Machines}.
Springer, New York.

\bibitem[Stone(1982)]{Stone1982}
Stone, C.J. (1982).
Optimal global rates of convergence for nonparametric regression.
\textit{Ann.\ Statist.} \textbf{10}, 1040--1053.

\bibitem[Szeg\"o(1975)]{Szego1975}
Szeg\"o, G. (1975).
\textit{Orthogonal Polynomials}, 4th ed.
Amer.\ Math.\ Soc., Providence.

\bibitem[Tsybakov(2009)]{Tsybakov2009}
Tsybakov, A.B. (2009).
\textit{Introduction to Nonparametric Estimation}.
Springer, New York.

\bibitem[van der Vaart and Wellner(1996)]{vanderVaart1996}
van der Vaart, A.W. and Wellner, J.A. (1996).
\textit{Weak Convergence and Empirical Processes}.
Springer, New York.

\bibitem[Villani(2009)]{Villani2009}
Villani, C. (2009).
\textit{Optimal Transport: Old and New}.
Springer, Berlin.

\bibitem[Wahba(1990)]{Wahba1990}
Wahba, G. (1990).
\textit{Spline Models for Observational Data}.
SIAM, Philadelphia.

\bibitem[Wand and Jones(1995)]{Wand1995}
Wand, M.P. and Jones, M.C. (1995).
\textit{Kernel Smoothing}.
Chapman \& Hall, London.

\bibitem[Wasserman(2006)]{Wasserman2006}
Wasserman, L. (2006).
\textit{All of Nonparametric Statistics}.
Springer, New York.

\bibitem[Yu(1997)]{Yu1997}
Yu, B. (1997).
Assouad, Fano, and Le Cam.
In: \textit{Festschrift for Lucien Le Cam}, Springer, 423--435.

\end{thebibliography}
\end{document}